\documentclass[11pt,a4paper]{article}
\usepackage{geometry}
\usepackage[english]{babel}
\usepackage[utf8]{inputenc}

\usepackage{csquotes}

\usepackage{booktabs}
\usepackage{multicol}

\usepackage{amsthm, amssymb}
\usepackage{thmtools, thm-restate}
\usepackage{mathtools}
\usepackage[mathscr]{euscript}
\usepackage{dsfont}

\usepackage[T1]{fontenc}
\usepackage{lmodern}
\usepackage{textcomp}
\usepackage{txfonts}

\usepackage{enumitem}	
\usepackage{microtype}

\usepackage{graphicx}
\usepackage[usenames]{xcolor}

\usepackage{pdfpages}

\usepackage{tikz}
\usetikzlibrary{cd}
\usetikzlibrary{patterns}
\usetikzlibrary{matrix}

\usepackage{hyperref} 

\usepackage{mleftright}
\renewcommand{\left}{\mleft}
\renewcommand{\right}{\mright}

\newcommand{\defeq}{\mathrel{\vcentcolon =}}

\newcommand{\eps}{\ensuremath{\epsilon}}

\newcommand{\N}{\ensuremath{\mathds N}}	

\newcommand{\Cpx}{\ensuremath{\mathds C}}
\newcommand{\C}{\ensuremath{\mathcal C}}
\newcommand{\B}{\ensuremath{\mathcal B}} 
\newcommand{\Cs}{\ensuremath{C^{\ast}}}

\newcommand{\Csalgebra}{\Cs-algebra}

\renewcommand{\star}{\ensuremath{^\ast}}

\newcommand{\staralgebra}{\star-algebra}
\renewcommand{\phi}{\ensuremath{\varphi}} 

\renewcommand{\ref}[1]{\textup{\ref{#1}}}

\usepackage[draft]{todonotes}   	

\newcommand{\etale}{\text{étale}}

\newcommand{\Csr}{\ensuremath{\Cs_{\text{r}}}} 

\newcommand{\sub}[1]{\ensuremath{\textup{Sub}(#1)}}
\newcommand{\Schreier}{\ensuremath{S_\Gamma^Q}}
\newcommand{\rlisom}{\ensuremath{\cong_{r,l}}}
\newcommand{\lockerZ}{\ensuremath{\mathbb{C}Z}}
\newcommand{\centralstar}{\infty}
\newcommand{\G}{\ensuremath{\mathcal{G}}}

\definecolor{thmback}{rgb}{0.93,0.94,0.95}
\definecolor{lightback}{rgb}{0.93,0.94,0.95}
\declaretheoremstyle[
	spaceabove=7pt, spacebelow=7pt,
	headfont=\normalfont\itshape,
	notefont=\mdseries, notebraces={(}{)},
	bodyfont=\normalfont,
	postheadspace=1em
]{rem}
\declaretheoremstyle[
	spaceabove=7pt, spacebelow=7pt,
	headfont=\normalfont\bfseries,
	notefont=\mdseries, notebraces={}{},
	bodyfont=\normalfont,
	postheadspace=1em,
	shaded={bgcolor=lightback,padding=2mm,textwidth=0.98\textwidth},
]{prop}
\declaretheoremstyle[
	spaceabove=20pt, spacebelow=7pt,
	headfont=\normalfont\bfseries,
	notefont=\mdseries, notebraces={(}{)},
	bodyfont=\normalfont,
	postheadspace=1em,
	shaded={bgcolor=thmback,padding=2mm,textwidth=0.98\textwidth},
]{forsty}
\declaretheoremstyle[
	spaceabove=7pt, spacebelow=7pt,
	headfont=\normalfont\bfseries,
	notefont=\normalfont\bfseries, 
	notebraces={}{},
	bodyfont=\normalfont,
	postheadspace=0.5em,
	headpunct={:}
]{def}
\newlength{\nameadjust}
\declaretheoremstyle[
	spaceabove=7pt, spacebelow=7pt,
	headfont=\normalfont\bfseries,
	notefont=\normalfont\bfseries,  
	notebraces={}{},
	bodyfont=\normalfont\itshape,
	postheadspace=1em,
	shaded={bgcolor=thmback,padding=2mm,textwidth=0.98\textwidth},
	headpunct={:},
]{thmsty}
\declaretheoremstyle[
	spaceabove=7pt, spacebelow=7pt,
	headfont=\normalfont\bfseries,
	notefont=\mdseries, notebraces={(}{)},
	bodyfont=\normalfont\itshape,
	postheadspace=1em,
	shaded={bgcolor=thmback,padding=2mm,textwidth=0.98\textwidth},
]{mainthmsty}
	\declaretheorem[style=thmsty, numberwithin=section,name=Theorem]{theorem}

	\declaretheorem[style=prop, numberlike=theorem,name=Proposition]{proposition}
	
	\declaretheorem[style=thmsty, numberlike=theorem,name=Corollary]{corollary}

	\declaretheorem[style=rem, numbered=no, name=Remark]{remark}

\title{A Groupoid Picture of Elek Algebras}
\date{\empty}
\author{Clemens Borys
\thanks{Supported by a grant from the Danish Council 
  for Independent Research, Natural Sciences.
}}

\begin{document}

\pagenumbering{arabic}
\setcounter{page}{1}

\maketitle

\begin{abstract}
  We describe a construction by G\'abor Elek,
  associating \Csalgebra{}s with uniformly recurrent subgroups,
  in the language of groupoid \Csalgebra{}s.
  This allows us to  simplify several proofs in the original paper
  and add a new characterisation of nuclearity.
  We furthermore relate our groupoids
  to the dynamics of the group acting on its uniformly recurrent subgroup.
\end{abstract}

\section{Introduction}
\label{sec:introduction}
The constructions of (reduced) \Csalgebra{}s
associated with groups
or, more generally,
crossed products associated with 
actions of groups on topological spaces
are well-known in the field of operator algebras
and continue to provide a handy tool 
for constructing and understanding
examples of \Csalgebra{}s,
as well as establishing ties to other branches of mathematics
such as dynamics
or geometric group theory.
Defined by Glasner and Weiss \cite{GlasnerWeissURS},
\emph{Uniformly recurrent subgroups},
or URS for short,
have recently drawn a lot of attention
in the world of \Csalgebra{}s,
when Kennedy \cite{KennedyIntrinsic}
managed to characterise \Cs-simplicity
of a discrete group $\Gamma$
as the absence of non-trivial amenable
uniformly recurrent subgroups.
Another relation between uniformly recurrent subgroups
and \Csalgebra{}s
is given by a construction of G.~Elek \cite{Elek_URS},
who constructs a \Csalgebra{} that is closely tied
to the dynamics of a finitely generated discrete group
acting on one of its uniformly recurrent subgroups $Z$,
but takes more of the combinatorial nature
of $Z$ into account
than the crossed product does.
This construction is thereby very well-suited
for finding \Csalgebra{}s with desired properties
by rephrasing such properties on the combinatorial level
of URSs described by their associated Schreier graphs,
and Elek obtains, for example,
a \Csalgebra{} with a uniformly amenable and a nonuniformly amenable trace.

In this paper we recast Elek's construction
from the viewpoint of groupoid \Csalgebra{}s,
simplifying and extending the ties
between properties of the URS
and its associated algebra.
Using the new angle, we improve on Elek's characterisation
of when his algebras are nuclear.

In addition to this introduction there are five sections.
After recalling the necessary terminology in Section \ref{sec:preliminaries},
we construct an \etale{} groupoid 
for a given URS $Z$
whose reduced \Csalgebra{}
is canonically isomorphic to Elek's \Csalgebra{} $\Csr(Z)$
in Section \ref{sec:identification}.
In Section \ref{sec:trafo}
we relate this groupoid to the dynamics
of the action of $\Gamma$ on $Z$ by conjugation.
Finally,
in Section \ref{sec:simNuc},
we use the new framework to give simpler proofs for some of Elek's results
on simplicity and nuclearity of the associated \Csalgebra{}s
and add the converse implication
to his characterisation of nuclearity of $\Csr(Z)$.

The author is grateful to his supervisors M.\ Musat and M.\ Rørdam
for their continued help and support,
as well as to G.\ Elek
for some exciting conversations.

\section{Preliminaries}
\label{sec:preliminaries}
We recall the definition of Elek's \Csalgebra{}s associated with uniformly recurrent subgroups.
Let $\Gamma$ be a finitely generated discrete group. 
Let $\sub{\Gamma}$ be the space of its subgroups, equipped with the 
topology of pointwise convergence
of the characteristic functions associated to the subsets
and left $\Gamma$-action by conjugation 
$\gamma.H = \gamma H\gamma^{-1}$ for $H\in \text{Sub}(\Gamma)$ and $\gamma \in \Gamma$.
For a discrete group $\Gamma$
this topology is also known as the Fell topology on $\sub{\Gamma}$.
Recall that a \emph{uniformly recurrent subgroup} or \emph{URS} $Z$ of $\Gamma$
is a closed, $\Gamma$-invariant subspace of $\sub{\Gamma}$
on which the action is \emph{minimal}, that is, on which every orbit is dense.
The URS is called \emph{generic}, if the stabiliser of any subgroup $H\in Z$
is as small as possible, namely $H$ itself.

Fixing a finite, symmetric system of generators $Q$ of $\Gamma$,
to each $H\in \sub\Gamma$ we may assign a rooted, labeled graph $\Schreier(H)$
called its \emph{Schreier graph},
which has vertex set $\Gamma/H$, root $H$,
and for every $\gamma H \in \Gamma/H$ and $q\in Q$ an edge 
from $\gamma H$ to $q \gamma H$ labeled by $q$.
We denote the shortest-path metric on a graph $S$ by $d$, 
or $d_S$ if there is ambiguity,
and likewise the balls of radius $R$ around a vertex $x \in S$ by $B_R(x)$ or $B_R(S,x)$.
On the space $\Schreier$ of Schreier graphs associated with URS's of $\Gamma$,
we introduce a metric by 
\begin{equation}
  \label{eq:SchreierMetric}
  d_{\Schreier}(S_1, S_2) \defeq 2^{-r}
\end{equation}
for $S_1 = \Schreier(H_1)$ and $S_2 = \Schreier(H_2)$ two graphs in $\Schreier$ 
and $r$ the largest integer
such that $B_r(S_1, H_1)$ and $B_r(S_2, H_2)$ are root-label isomorphic.
This space carries a left $\Gamma$-action, where 
$\gamma.\Schreier(H) = \Schreier(\gamma H \gamma^{-1})$,
that is, $\gamma$ acts by changing the root.
Elek identifies the graphs $\Schreier(H)$ 
for which the orbit closure of $H$ forms a URS 
as those where any root-label isomorphism class of balls is repeated
with at most bounded distance from any point in the graph 
(see \cite[Proposition 2.1]{Elek_URS}),
and the graphs for which it forms a generic URS 
as those where vertices with large isomorphic balls are sufficiently far apart
(see \cite[Proposition 2.3]{Elek_URS}).

To associate a \Csalgebra{} with a given URS $Z$, 
Elek considers its \emph{local kernel algebra} $\lockerZ$ 
formed by the \emph{local} kernels 
$K\colon \Gamma/H \times \Gamma/H \rightarrow \Cpx$ of finite width 
on $\Schreier(H)$ for some subgroup $H\in Z$.
A kernel $K$ on $\Schreier(H)$ is of width $R$, if $K(x,y) = 0$
for any two vertices $x,y \in \Schreier(H)$ with $d(x,y) > R$.
It is furthermore \emph{local} with width $R$, 
if $K(x, \gamma.x) = K(y, \gamma.y)$ for any $x$ and $y$ with root-label isomorphic $R$-balls
and $\gamma \in \Gamma$ of length at most $R$.
Equipped with pointwise addition, 
convolution $KL(x,y)=\sum_z K(x,z)L(z,y)$, 
and involution $K^\ast(x,y) = \overline{K(y,x)}$
for local kernels $K$ and $L$ and $x,y,$ and $z$ in $\Gamma / H$,
the local kernel algebra forms a \staralgebra{}.
Up to isomorphism, this algebra does not depend on the choice of root $H\in Z$.
A ``regular'' representation of $\lockerZ$ on $\ell^2(\Gamma/H)$ 
is given by $(Kf)(x) = \sum K(x,y) f(y)$ 
for $f\in \ell^2(\Gamma/H)$.
The reduced $\Cs$-algebra $\Csr(Z)$ of $Z$ 
is the completion of $\lockerZ$ in the norm induced by this representation.

Several \Csalgebra{}ic properties of $\Csr(Z)$ can be read off of the URS $Z$
and its Schreier graph.
Genericity of $Z$ implies simplicity of $\Csr(Z)$,
as discussed in Section \ref{sec:simNuc}.
Furthermore, a \emph{local} version of Yu's property $A$ for the Schreier graph
$\Schreier(H)$,
for any subgroup $H$ in the URS $Z$,
is equivalent to nuclearity of $\Csr(Z)$.
Recall that a Schreier graph $\Schreier(H)$ for $H\in Z\subseteq \Gamma$  
has Elek's \emph{local property $A$},
if there is a sequence of \emph{local} functions
$\rho^n\colon \Gamma/H \rightarrow l^2(\Gamma/H), x \mapsto \rho^n_x$,
such that $\|\rho^n_x\|_2=1$, 
while $d(x,y)\leq n$ implies $\|\rho^n_x - \rho^n_y\|\leq 1/n$.
As with kernels, \emph{locality} of $\rho^n$
means that there is $R_n>0$ such that $\rho^n_x$ is supported in the 
$R_n$-ball $B_{R_n}(x)$ centred at $x$ 
and whenever $\theta$ is a root-label isomorphism $B_{R_n}(x) \rightarrow B_{R_n}(y)$, 
we have $\rho^n_y \circ \theta = \rho^n_x$.

\section{URS Algebras as Groupoid Algebras}
\label{sec:identification}
For a given uniformly recurrent subgroup,
we proceed to construct a groupoid 
whose regular representation and reduced \Csalgebra{}
model Elek's construction on the local kernel algebra.

Let $Z$ be a uniformly recurrent subgroup of a discrete group $\Gamma$,
fix $H\in Z$,
and let $S= \Schreier(H)$ be its Schreier graph.
For any $n\in\N$ we define an equivalence relation on the vertices $V(S)$ of $S$
by 
\[p\sim_n q \Leftrightarrow B_n(S,p) \cong_{r,l} B_n(S,q),\]
that is, if the $n$-balls around $p$ and $q$ are isomorphic 
under an isomorphism preserving the roots and labels.
Such a root-label isomorphism is necessarily unique.
Equivalently, if $\gamma_p$ and $\gamma_q$ are elements of $\Gamma$
describing paths from the root of $S$ to $p$ and $q$, respectively,
then $p \sim_n q$ if and only if 
$\gamma_p.S$ and $\gamma_q.S$ are $2^{-n}$-close
in the metric $d_{\Schreier}$ introduced on $\Schreier$ 
in Equation \eqref{eq:SchreierMetric}.
Let $E_n = V(S)/{\sim_n}$ denote the finite set of equivalence classes of ${\sim_n}$ 
equipped with the discrete topology
and the obvious connecting maps $e_{n+1}\colon E_{n+1}\rightarrow E_n$.
Then, as in \cite[Lemma 6.1.4]{Elek_URS}, it is easy to check that 
$\varprojlim E_n$ is homeomorphic to $Z$ as a subspace of $\sub{\Gamma}$,
which in turn is a Cantor space or a finite discrete set.
It is noteworthy that under this identification
the orbit of $H$ in $Z$ 
is exactly described by those elements in $\varprojlim E_n$,
which can be represented by the equivalence classes $[p]_n$
of a \emph{fixed} vertex $p \in S$.
The other elements 
describe subgroups in the orbit closure,
but not the orbit,
of $H$.
All elements of $\G^x$ are therefore represented by
We employ this description of the space $Z$ 
to construct an ample Hausdorff \etale{} groupoid $\G$ 
with unit space $\G^0$ homeomorphic to $Z$, 
whose reduced groupoid $\Cs$-algebra is $\Csr(Z)$.

Let $\G^0$ be given by $\varprojlim E_n$ and
let $x=([x_0]_0, [x_1]_1, \ldots) \in \G^0$.
In slight abuse of notation we will continue to write $x_n$ for a representing vertex
of the class $[x_n]_n$ in $E_n$ 
that forms the $n$-coordinate of $x$.
The arrows of $\G^x$ will be given by equivalence classes 
of pairs $(x,\gamma)$ for $\gamma\in\Gamma$,
where we identify two such pairs $(x, \gamma)$ and $(x, \gamma')$ 
for $l(\gamma')\geq l(\gamma)$, 
if $\gamma x_{l(\gamma')} = \gamma' x_{l(\gamma')}$.
Note that $d(x_{l(\gamma)}, \gamma x_{l(\gamma)})$ 
might be strictly less that $l(\gamma)$, 
in which case there is another $\gamma'\in \Gamma$ 
of length $d(x_{l(\gamma)}, \gamma x_{l(\gamma)})$,
such that $(x, \gamma)$ and $(x, \gamma')$ denote the same arrow.
We call such $\gamma'$ of minimal length.
The range map is consequently defined as $r(x, \gamma) = x$.
We fix the source of $(x, \gamma)$ as 
$\gamma x \defeq ([\gamma x_{l(\gamma)}]_0, [\gamma x_{l(\gamma)+1}]_1, \ldots)$.
Intuitively, an arrow with range $x$ is thought of as a path from $x$ to $\gamma x$
seen as vertices in the Schreier graph,
but $\gamma x_n$ is only well-defined for $n\geq l(\gamma)$, 
in which case $\gamma x_n$ determines 
a unique isomorphism class of $(n-d)$-balls
and thereby a class in $E_{n-d}$ 
where $d=d(x_n, \gamma x_n)$
.
All elements of $\G^x$ are therefore represented by
a pair $(x, \gamma)$ with $\gamma$ of minimal length
and we can describe $(x,\gamma)$
as consistent choices of vertices $\gamma x_n$ 
in the balls described by the classes $[x_n]_n$.

More formally,
we define $\G$ as a subset of the projective limit $\varprojlim F_n$
of the finite, discrete sets
\begin{align}
  F_n &= 
  \bigsqcup\limits_{[x_n] \in E_n} B_n(S,x_n) \sqcup \{\centralstar_n\}
  \label{eq:Fn}
  \\
  &\cong 
  \bigsqcup\limits_{[x_n] \in E_n} 
  \left\{ 
    \gamma \in \Gamma \mid l(\gamma) \leq n 
  \right\}/{\approx_{x_n}}
  \sqcup \{\centralstar_n\}
  \label{eq:FnAlternative}
  ,
\end{align}
where $B_n(S,x_n)$ denotes the $n$-ball in $S$ 
that is determined uniquely by $[x_n]$,
even if there is a choice in the representing vertex $x_n$.

To avoid this choice of representing elements,
a pair $([x_n], y)$ with $y\in B_n(S, x_n)$ as in Equation \eqref{eq:Fn}
can be more readily expressed as a pair
$([x_n], \gamma)$ as in Equation \eqref{eq:FnAlternative},
where $\gamma$ is any chosen path
from $x_n$ to $y$ inside $B_n(S, x_n)$,
up to $\gamma \mathbin{\approx_{x_n}} \gamma'$
if both paths lead to the same vertex in $S$,
that is, if $\gamma.x_n = \gamma'.x_n$.
Implicit in this description is our later identification of $\G$
with a quotient of the transformation groupoid $Z \rtimes \Gamma$.
In this picture, $\centralstar_n$ fills in for 
choices of vertices that are not contained in $B_n(S, x_n)$
or respectively for $\gamma$ whose lengths as words in the generators
exceeds $n$.
The connecting maps $f_{n+1}\colon F_{n+1} \rightarrow F_n$ are then given by 
\begin{align*}
  ([x_{n+1}], \gamma) &\mapsto 
  \begin{cases}
    (e_{n+1}([x_{n+1}]), \gamma) 
    &\text{if } d(x_{n+1}, \gamma x_{n+1}) \leq n\\
    \centralstar_n &\text{else}
  \end{cases}\\
  \centralstar_{n+1} &\mapsto \centralstar_n
  .
\end{align*}

Now let $\G \defeq \varprojlim F_n \setminus \{\centralstar\}$ 
with $\centralstar=(\centralstar_1, \centralstar_2, \ldots)$, 
equipped with the subspace topology of the projective limit.
Equivalently, 
with $x$ and $y$ in $\G^{0}$
and $\gamma, \gamma' \in \Gamma$,
the topology on $\G$ is given by the metric
$d_\G((x, \gamma),(y, \gamma')) = 2^{-N}$ 
for $N$ maximal such that 
$([x_N]_N, \gamma.x_N)$ and $([y_N]_N, \gamma.y_N)$ coincide in $F_N$. 
To simplify notation, 
we write $(x, \gamma)$ in place of the equivalence class it represents
and identify $([x_n], \gamma)$ with $\centralstar_n$, whenever $l(\gamma) > n$.

If $(x, \gamma)$ and $(y, \gamma')$ are composable,
that is, if $x = \gamma'.y$, 
then we define the composition $(y,\gamma')(x,\gamma)$ 
to be $(y, \gamma\gamma')$.
Note, that $s((y,\gamma\gamma')) = s(x, \gamma')$, 
since $([x_0]_0, [x_1]_1, \ldots)$
coincides with $([x_N]_0, [x_{N+1}]_1, \ldots)$ 
in $\G^0$ for any $N\in\N$.
We see that any $x\in \G^0$ is a unit in $\G$ 
when written as $(x,e)$ for $e\in \Gamma$ the neutral element, 
and consequently the equivalence class represented by $(x,\gamma)$ 
has as inverse  the class represented by $(\gamma.x, \gamma^{-1})$.
\begin{proposition}
  Let $Z$ be a uniformly recurrent subgroup of a discrete group $\Gamma$.
  The above turns $\G$ into an
  ample minimal Hausdorff \etale{} groupoid 
  with unit space homeomorphic to $Z$.
\end{proposition}
\begin{proof} 
  It is easy to see that the operations above indeed turn $\G$ into a groupoid.
  Equipping $\G$ with the locally compact Hausdorff subspace topology of 
  $\varprojlim F_n$,
  we have to check that the defined operations are continuous.
  Consider the basis open sets 
  \[
    U_{e_N, \gamma} = \left\{
      (x, \gamma) \in \G \mid
      [x_N]_N = e_N
    \right\},
  \]
  that fix $\gamma \in \Gamma$
  and
  $e_N \in E_N$ for some $N$.
  The range map is obviously continuous,
  as any basic open set of $\varprojlim E_n$
  can be turned into a union of basic open sets of $\varprojlim F_n$
  by letting $\gamma$ vary.
  Inversion is continuous, since the action of $\Gamma$ on $Z$ is,
  which in turn makes the source map continuous.
  To see that the composition is continuous, 
  we fix $e_N \in E_N$ and $\gamma \in \Gamma$
  and find the preimage of $U_{e_N, \gamma}$ under the composition map.
  For $(x,\eta)(\eta^{-1}.x, \eta')$
  we first need that $x \in U_{e_N, \eta}$
  and secondly that $\eta' = \eta^{-1}\gamma$.
  Hence the desired preimage is described by the intersection of
  \[
    \bigcup\limits_{\eta \in \Gamma}
    U_{e_N, \eta}
    \times
    U_{e_0, \eta^{-1}\gamma}
  \]
  with the subspace $\G^{(2)}$ of composable pairs in $\G \times \G$
  and therefore open in $\G^{(2)}$.
  We used $e_0$ to denote the unique equivalence class in $E_0$.

  To see that the range map is a local homeomorphism,
  note that it restricts to a homeomorphism onto its image
  on every basic open set $U_{e_N, \gamma}$,
  and $\G$ is therefore \etale{}.
  Finally, $\G$ is ample as it is an \etale{} groupoid
  with totally disconnected unit space.

  As the orbits of $Z$ and $\G^{(0)}$ coincide, $\G$ is minimal.
\end{proof}

It is noteworthy that the construction above 
does not depend on the choice of $H\in Z$:
\begin{proposition}
  For different choices of $H,H'\in Z$ the groupoids $\G,\G'$ constructed above are isomorphic.
\end{proposition}

\begin{proof}
  This amounts to showing that the root-label equivalence classes 
  $E_n$ and $E_n'$ of $n$-balls in $S=S_\Gamma^Q(H)$ and $S'=S_\Gamma^Q(H')$ 
  are identical sets of rooted, labeled balls.
  That is, for every $p\in V(S)$ there is $p'\in V(S')$ 
  such that $B_n(S,p) \rlisom B_n(S',p')$.
  But as $Z$ is uniformly recurrent, 
  for $p=\gamma H$ the subgroup $\gamma H \gamma^{-1}$ is in the orbit closure of $H'$
  and therefore $S_\Gamma^Q(\gamma H \gamma^{-1})$ is in the orbit closure of $S'$.
  Hence there is $\gamma'\in\Gamma$ such that 
  $B_n(S, p) 
  \rlisom B_n(S_\Gamma^Q(\gamma H \gamma^{-1}), \gamma H \gamma^{-1}) 
  \rlisom B_n(S', \gamma'H')$
  and $p'=\gamma'H'$ gives $E_n\subseteq E_n'$.
  Equality follows by symmetry.
\end{proof}

Next, we identify the local kernel algebra $\lockerZ$ of $Z$ with a dense subset of $\C_c(\G)$.
Given a kernel $K\in \lockerZ$ as a function on $V(S) \times V(S)$, 
recall that there is a minimal $N\in \N$ called the width of $K$ 
such that $K$ vanishes on any pairs $(x,y)$
where the distance of $x$ and $y$ is more than $N$
and such that $B_N(S,x)\rlisom B_N(S,y)$
implies that $K(x,\gamma x) = K(y,\gamma y)$ 
if $l(\gamma)\leq N$,
so that $K$ only depends on the root-label isomorphism class of $N$-balls.
To $K$ we assign a function $f_K\in\C_c(\G)$ 
by $f_K((x,\gamma)) = K(x_M, \gamma x_M)$ 
with $M=\max\{N, l(\gamma)\}$.
If $\gamma$ is chosen of minimal length, we may pick $M=N$.
Equivalently, $f_K$ evaluates $(x,\gamma)$ at it's component $([x_N]_N, \gamma)$
in $F_N$ and assigns $K(x_N, \gamma x_N)$, 
where $K(\centralstar_N) = 0$.
This is well-defined, 
as the vertex $x_N$ is given up to root-label isomorphism
of $N$-balls.
The function $f_K$ is continuous, 
because it is uniformly locally constant:
It is constant on any $2^{-N}$-ball in $\G$.
To see that $f_K$ is compactly supported, 
note that the embedding $\G \hookrightarrow \varprojlim F_n$ 
is the one-point compactification of $\G$.
Therefore, a set $U\subseteq \G$ is relatively compact, 
exactly if there is $M\in\N$ 
such that no element of $U$ has $F_M$-component $\centralstar_M$.
Equivalently, $U$ as a subset of $\varprojlim F_n$ 
does not intersect the $2^{-M}$-ball centred at $\centralstar$.
As $f_K$ is supported outside of the $2^{-N}$-ball centred at $\centralstar$ 
for $N$ the width of $K$, 
it is compactly supported.

Conversely, 
any locally constant function $f\in \C_c(\G)$ defines a kernel $K_f \in \lockerZ$.
As $f$ is locally constant, for each $g\in \G$ we can pick a Ball centred at $g$, 
on which $f$ is constant. 
Then finitely many of such balls cover the support of $f$ 
and we may pick $N\in \N$ such that these have radius at least $2^{-N}$.
Since two $2^{-N}$-balls in $\G$ or $\varprojlim F_n$ are either disjoint or equal, 
$f$ is constant on any $2^{-N}$-ball 
and supported outside of the $2^{-N}$-ball of $\centralstar$.
Given such $f$, we may now define a local kernel $K_f$ with width at most $N$ as follows:
For a fixed vertex $p \in V(S)$ 
consider the unit $[[p]] \coloneq ([p]_0, [p]_1, \ldots) \in \G^0$ 
and define $K_f(p, \gamma p) \defeq f([[p]], \gamma)$.
If $\gamma p = \gamma'p$, then $([[p]], \gamma) = ([[p]], \gamma')$,
so $K_f$ is a well-defined kernel.
As the $F_N$-component of $([[p]], \gamma)$ is $\centralstar_N$ if $d(p, \gamma p)>N$, 
we have $K_f(p, \gamma p)=0$ in that case.
Furthermore, if the $N$-balls $B_N(S,p)$ and $B_N(S,q)$ are root-label-isomorphic 
for $p,q\in V(S)$ and $l(\gamma)\leq N$, 
then $([[p]], \gamma)$ and $([[q]], \gamma)$ are $2^{-N}$-close, 
since their first $N$ components coincide.
As $f$ is constant on $2^{-N}$-balls, 
we have $K_f(p,\gamma p) = K_f(q,\gamma q)$, 
so $K_f$ is local of width $N$.

It is easy to check that $f_{K_f}=f$ and $K_{f_K}=K$, 
so the local kernel algebra $\lockerZ$ of $Z$ is in bijection 
with the subset of locally constant functions in $\C_c(\G)$,
which is dense in $\C_c(\G)$, as $\G$ is totally disconnected.
As the \staralgebra{} structure of $\lockerZ$ is preserved under this inclusion, 
we identify $\lockerZ$ with a \star{}-subalgebra of $\C_c(\G)$:
For $p \in V(S)$ and $\gamma\in \Gamma$ we have $s([[p]], \gamma) = [[\gamma.p]]$ 
and the orbit of $[[p]]$ in $\G$ is $\{[[q]] \mid q\in V(S)\}$.
We calculate
\begin{align*}
  f_{L\ast K}([[p]],\gamma) 
  &= L\ast K(p,\gamma p) 
  = \sum\limits_{q\in V(S)} 
  L(p,q) K(q,\gamma p)
  \\
  &= \sum\limits_{([[p]], \gamma') \in \G^{[[p]]}} 
  f_L([[p]], \gamma') f_K([[\gamma'p]], \gamma(\gamma')^{-1})
  = f_L \ast f_K([[p]],\gamma)
  \intertext{and}
  (f_K)^\ast([[p]],\gamma) 
  &= \overline{f_K([[\gamma p]], \gamma^{-1})} 
  = \overline{K(\gamma p, p)} 
  = K^\ast(p, \gamma p) 
  = f_{K^\ast}([[p]], \gamma).
\end{align*}
We conclude that $f_{L\ast K} = f_L \ast f_K$ and $(f_K)^\ast = f_{K^\ast}$ 
on the subset $\{([[p]], \gamma) \mid p\in V(S), \gamma \in \Gamma\}\subseteq \G$
of arrows in $\G$, whose range is given 
by the equivalence classes of a single, constant vertex $p$ in $S$.
This subset is dense, as the orbit closure of $S$ is dense in $S_\Gamma^Q(Z)$ 
and as the functions in question are continuous,
they coincide on all of $\G$.

Next, we assert that $\lockerZ$ is dense in $\C_c(\G)$ even in reduced norm 
and that the reduced norm on $\lockerZ\subseteq \C_c(\G)$
coincides with the norm obtained by representing $\lockerZ$ on $\B(l^2(\Gamma/H))$.

\begin{theorem}
  Let $Z$ be a uniformly recurrent subgroup 
  of a finitely generated group $\Gamma$
  and $\G$ as above.
  The $\Cs$-algebras $\Csr(Z)$ and $\Csr(\G)$ are isomorphic
  with the isomorphism extending the canonical construction above uniquely.
\end{theorem}
Recall that for an \etale{} groupoid $\G$ 
the reduced norm on $\C_c(\G)$ is given as $\|f\| = \sup_x \|\pi_x(f)\|$ 
with the representations $\pi_x\colon \C_c(\G) \rightarrow \B(l^2(\G_x))$ 
by \[\pi_x(f)\delta_g = \sum\limits_{g'\in \G_{r(g)}} f(g') \delta_{g'g}\] 
for $x\in \G^0$ and $g\in \G_x$.
The (faithful) representation $\pi$ of $\lockerZ$ on $\B(l^2(\Gamma/H))$ 
is given by 
\[\pi(K)\delta_p = \sum\limits_{q\in V(S)} K(q,p) \delta_q.\]

\begin{proof}
  Let us first consider $x=[[H]]$, the unit represented by the root in $S$. 
  We obtain a map $V(S) \rightarrow \G_{[[H]]}$ 
  by $\gamma H \mapsto ([[\gamma H]], \gamma^{-1})$,
  mapping a vertex $\gamma H$ in $S$
  to the arrow in $\G$ that is described by any path from $H$ to $\gamma H$.
  This is obviously surjective and is well-defined, 
  as the arrows $([[\gamma H]], \gamma^{-1})$ and $([[\gamma' H]], (\gamma')^{-1})$ 
  are identified if $\gamma H=\gamma' H$.
  It is furthermore injective: 
  If $([[\gamma H]], \gamma^{-1}) = ([[\gamma' H]], (\gamma')^{-1})$, 
  then the $N$-balls centred at $\gamma H$ and $\gamma' H$ 
  are root-label isomorphic 
  with the isomorphism mapping $H$ to $H$
  if $N> l(\gamma) + l(\gamma')$.
  But a root-label-isomorphism of $N$-balls in Schreier graphs 
  is the identity if it has a fixed point,
  hence $\gamma H = \gamma' H$.
  We have thus established a bijection between $\G_{[[H]]}$ and $\Gamma/H = V(S)$.
  This yields a unitary $T\colon l^2(V(S)) \rightarrow l^2(\G_{[[H]]})$ 
  which intertwines the representations $\pi$ and $\pi_{[[H]]}$ on $\lockerZ$:
  Let $h\in l^2(V(S))$. 
  Then
  \begin{align*}
    \pi(K) h (\gamma H) 
    &= \sum\limits_{\gamma' H \in V(S)} K(\gamma H, \gamma' H) h(\gamma' H)
    \intertext{and so}
    T\left( \pi(K) h\right)([[\gamma H]], \gamma^{-1})
    &= \sum\limits_{\gamma' H \in V(S)} 
    K(\gamma H, \gamma' H) h(\gamma' H)
    = \sum\limits_{\gamma' H \in V(S)} 
    f_K([[\gamma H]], \gamma' \gamma^{-1}) h(\gamma' H)
    \\
    &= \sum\limits_{\gamma' H \in V(S)} 
    f_K([[\gamma H]], \gamma' \gamma^{-1}) (Th)([[\gamma' H]], (\gamma')^{-1})
    \\
    &= \sum\limits_{g\in \G_{[[H]]}} 
    f_k(([[\gamma H]], \gamma^{-1})g^{-1}) (Th)(g)
    = \left( \pi_{[[H]]}(f_K) Th \right)([[\gamma H]], \gamma^{-1}),
  \end{align*}
  and therefore the representations $\pi$ and $\pi_{[[H]]}$ define identical reduced norms 
  on $\lockerZ\subseteq \C_c(\G)$.

  Morally, this already implies 
  that all source-fibre representations $\pi_x$ for $x\in \G^0$ 
  are unitarily equivalent to $\pi$,
  since the groupoid $\G$ does not depend on the choice of $H\in Z$ in its construction 
  and for every $x\in \G^0$ there is a unique subgroup in $Z$
  which is mapped to $x$ under the homeomorphism $Z\cong \G^0$.
  For completeness we nevertheless show 
  that all reduced representations of $\G$ are equivalent.
  As in any groupoid, the representations $\pi_x$ are unitarily equivalent 
  for any two units $x$ that share an orbit. 
  Therefore we only need to consider $\pi_x$ for $x\in \G^0$ 
  that corresponds to a subgroup $H'$ in $Z\setminus \Gamma.H$.
  For fixed $K\in \lockerZ$ of width $R$ and $\eps>0$, 
  we may pick $h\in l^2(\G_{[[H]]})$ of norm one 
  such that $\|\pi_{[[H]]}(f_K)h\|_2 > \|\pi_{[[H]]}\left( f_K \right)\|-\eps$
  and $h$ is supported on arrows $([[\gamma H]], \gamma^{-1})$ 
  with $l(\gamma) \leq N$ for some $N\in \N$.
  As the orbit of $H'$ in $Z$ is dense, 
  there is a unit $y$ in the orbit of $x$ in $\G^0$,
  such that $y$ and $[[H]]$ are $2^{-(N+R)}$-close,
  that is, their $F_{N+R}$-components coincide.
  Hence, for every arrow $([[\gamma H]], \gamma^{-1}) \in \G_{[[H]]}$ 
  with $l(\gamma) \leq N+R$
  there is a \emph{unique} arrow $(\gamma.y, \gamma^{-1}) \in \G_{y}$
  that is described by the same $\gamma^{-1}$,
  since two paths of length less than $N+R$ starting at $H$
  end in the same vertex,
  exactly if the analogous paths in the isomorphic $N+R$-ball of $y$ do.
  This yields a bijection between the subspaces of $\G_{[[H]]}$ and $\G_y$
  described by elements of $\Gamma$ with length at most $N+R$.
  Extending by zero, 
  we transport $h\in l^2(\G_{[[H]]})$ to a function $h' \in l^2(\G_y)$ 
  with $1 = \|h\|_2 = \|h'\|_2$
  along this bijection.
  Noting that $\pi_{[[H]]}(f_K)h \in l^2(\G_{[[H]]})$ 
  is supported on $([[\gamma H]], \gamma^{-1})$ with $l(\gamma) \leq N+R$, 
  we may likewise transport this to a function in $l^2(\G_y)$ of the same norm.
  It is now easy to see that this function will just be $\pi_y(f_K)h'$,
  whence $\|\pi_y(f_K)\| \geq \|\pi_y(f_K)h'\|_2 > \|\pi_{[[H]]}(f_K)\| - \eps$
  and by unitary equivalence $\|\pi_x(f_K)\| > \|\pi_{[[H]]}(f_K)\| - \eps$.
  By symmetry we obtain the converse direction,
  implying that all norms on $\C_c(\G)$ induced by reduced representations are identical
  and coincide on $\lockerZ \subseteq \C_c(\G)$ with the reduced norm of $\lockerZ$.

  We finally show that $\lockerZ$ is dense in $\C_c(\G)$ in this reduced norm, 
  whence $\Csr(Z) \cong \Csr(\G)$.
  Let $f\in \C_c(\G)$. 
  As $f$ is compactly supported, it vanishes on the
  $2^{-R}$-ball around $\centralstar \in \varprojlim F_n$ for some $R$,
  so that $f([[\eta H]], \gamma^{-1}) = 0$ 
  if $\gamma$ is chosen of minimal length and yet $l(\gamma)>R$.
  We therefore find that
  \begin{align*}
    \|\pi_{[[H]]}(f)h\|^2_2 
    &= \sum\limits_{\gamma H \in \Gamma/H} 
    \left|\left(\pi_{[[H]]}(f)h\right)([[\gamma H]], \gamma^{-1})\right|^2\\
    &= \sum\limits_{\gamma H \in \Gamma/H} 
      \bigg|
        \sum\limits_{ \gamma'H \in B_R(S, \gamma H)} 
          f([[\gamma H]], \gamma'\gamma^{-1}) h([[\gamma'H]], (\gamma')^{-1})
      \bigg|^2\\
    &\leq \|f\|_\infty^2 \sum\limits_{\gamma H \in \Gamma/H} 
      \bigg|
        \sum\limits_{ \gamma'H \in B_R(S, \gamma H)} h([[\gamma'H]], (\gamma')^{-1})
      \bigg|^2\\
    &\leq \|f\|_\infty^2 
    \big\| \sum\limits_{\eta\in\Gamma,\,l(\eta)\leq R} h \big\|^2_2
    \leq (|Q|+1)^{R} \|f\|_\infty^2 \|h\|_2^2,
  \end{align*}
  as there are less than $(|Q| + 1)^R$ elements
  of length at most $R$ in $\Gamma$.
  Hence $\|f\|_r \leq (|Q|+1)^{R} \|f\|_\infty$ 
  for all $f\in \C_c(\G)$ 
  supported outside of the $2^{-R}$-ball around $\centralstar$
  with $\|f\|_r = \|\pi_{[[H]]}(f)\|$ the unique reduced norm.
  For any $\epsilon>0$, 
  by partitioning $\G$ into open $2^{-N}$-balls for large $N$,
  we may approximate $f$ in $\C_c(\G)$ up to $\eps$ by a function $f_\eps$ 
  that is constant on every $2^{-N}$-ball.
  As two balls are either disjoint or one is contained in the other,
  we may choose $f_\eps(g)$ supported outside of the $2^{-R}$-ball of $\centralstar$
  for large $N$,
  such that 
  $\|f-f_\eps\|_r \leq (|Q|+1)^{R} \|f-f_\eps\|_\infty 
  = \eps(|Q|+1)^{R}$.
\end{proof}

To recap, for any uniformly recurrent subgroup $Z$ 
we have constructed an ample minimal \etale{} Hausdorff groupoid 
with unit space homeomorphic to $Z$, 
such that the reduced $\Cs$-algebras $\Csr(Z)$ of $Z$ and $\Csr(\G)$ of $\G$ coincide.
This enables us to examine $\Csr(Z)$ using tools for groupoid $\Cs$-algebras.

\section{Relation to the Transformation Groupoid}
\label{sec:trafo}
In this section we shed some light on the relationship
between our newly defined groupoid $\G$
associated with a uniformly recurrent subgroup $Z$ of $\Gamma$
and the transformation groupoid $Z \rtimes \Gamma$
associated with the action of $\Gamma$ on $Z$ by conjugation.

As a space, 
the transformation groupoid $Z \rtimes \Gamma$
is simply the cartesian product $Z \times \Gamma$
equipped with the product topology.
The unit space is given by the subspace $Z \times \{e\}$
and identified with $Z$,
the range and source of an arrow $(H, \gamma)$
are respectively given by $H$ and $\gamma^{-1}.H$,
while the product of two composable arrows is
$(H, \gamma) (\gamma^{-1}.H, \eta) = (H, \gamma\eta)$.
This turns $Z\rtimes \Gamma$ into a Hausdorff \etale{} groupoid,
which for a URS $Z$ is furthermore ample and minimal.

To distinguish our groupoid $\G$ from $Z\rtimes \Gamma$,
we describe the range fibres $\G^H$ further.
Recall that the homeomorphism between $Z$ and $\varprojlim E_n$
describes every group $H\in Z$ 
as a sequence of isomorphism classes of balls in $\Schreier(H')$
for any $H' \in \Gamma$.
In particular, the isomorphism class in $E_n$ associated with $H$
is given by the $n$-ball around the root in $\Schreier(H)$.
Two arrows $(H,\gamma)$ and $(H,\eta)$ in $\G$ 
will then coincide 
if and only if the paths in $\Schreier(H)$
that start at the root
and are described by $\gamma$ and $\eta$
end in the same vertex
so that $\gamma H = \eta H$
or equivalently $\eta^{-1}\gamma \in H$.
By this identification
we obtain a surjective map
$q\colon Z \rtimes \Gamma \rightarrow \G$
that maps 
$(H,\gamma) \mapsto ([[H]], \gamma)$.
To simplify this notation 
and remove the ambiguity in the choice of $\gamma$,
we continue to denote $([[H]], \gamma)$
as $(H, \gamma H)$.
We see that the range fibres $\G^H$
are given by the quotients $\Gamma/H$
while those of $(Z\rtimes \Gamma)^H$
are simply given by $\Gamma$
and $\G$ arises from the transformation groupoid
describing the action of $\Gamma$ on $Z$
without taking locality into account
by dividing out the appropriate subgroup in every fibre.

Indeed, 
this identification is not merely as sets,
but as topological groupoids:
\begin{proposition}
  The map $q\colon Z\rtimes \Gamma \rightarrow \G$
  given by $(H, \gamma) \mapsto (H, \gamma H)$
  for a subgroup $H\in Z$ and a group element $\gamma\in \Gamma$
  is a continuous, open, and surjective groupoid homomorphism.
  In particular,
  $\G$ is a quotient of the transformation groupoid $Z \rtimes \Gamma$.
\end{proposition}
\begin{proof}
  As argued above, $q$ is surjective and it is obviously a homomorphism.

  We fix a basis of the topology of $\G$ as
  \[ 
    U_{H,N,\gamma} =
    \left\{ 
      (K, \gamma K) 
      \mid
      d_{\Schreier}(\Schreier(H), \Schreier(K)) \leq 2^{-N}
    \right\}
  \]
  indexed by a subgroup $H\in Z$, 
  an element $\gamma\in \Gamma$
  and $N\in \N$ with $l(\gamma) \leq N$
  and consisting of all arrows $(K, \gamma K)$
  such that the $N$-balls around the root
  in $\Schreier(H)$ and $\Schreier(K)$ coincide.
  Equivalently, a group element $\eta \in \Gamma$
  with $l(\eta) \leq 2N$
  is contained in $K$
  if and only if it is contained in $H$.
  Likewise, we fix a basis
  \[
    V_{H,N,\gamma} = 
    \left\{ 
      (K, \gamma) 
      \mid
      d_{\Schreier}(\Schreier(H), \Schreier(K)) \leq 2^{-N}
    \right\}
  \]
  of the topology of $Z \rtimes \Gamma$.

  It is now easy to see that $q$ is open.
  Let $(K, \gamma K) \in q(V_{H,N,\gamma})$.
  Then $U_{K, M, \gamma}$
  for $M = \max\{N, l(\gamma)\}$
  is a neighbourhood of $(K, \gamma K)$
  contained in $q(V_{H, N, \gamma})$.

  Conversely,
  \[
    q^{-1}(U_{H,N,\gamma}) =
    \left\{
      (K, \gamma k) \mid
      d_{\Schreier}(\Schreier(H), \Schreier(K)) \leq 2^{-N},\,
      k\in K
    \right\},
  \]
  and we claim that this is open in $Z\rtimes \Gamma$.
  Indeed,
  for every $(K,\gamma k) \in q^{-1}(U_{H,N,\gamma})$
  the neighbourhood $V_{K,M,\gamma k}$
  is contained in $q^{-1}(U_{H,N,\gamma})$,
  where $M = \max\{N, l(k)\}$,
  so that any subgroup in the $2^{-M}$-ball around $K$
  is guaranteed to contain the element $k$.
\end{proof}

\section{Simplicity and Nuclearity}
\label{sec:simNuc}

We employ our description of the Elek algebras as groupoid algebras
to give simplified proofs of Elek's characterisations,
explaining why they arise in the language of groupoids.
\begin{proposition}\label{prop:generic-principal}
	If $Z$ is generic, then $\G$ is principal.
\end{proposition}
\begin{proof}
  Suppose $\G$ is not principal.
  Then there is a unit $x=([x_0]_0, [x_1]_1, \ldots)$
  and an arrow in the isotropy $\G_x^x$ that is not a unit.
  Therefore there is a group element $\gamma \in \Gamma$ 
  such that $x = \gamma.x$ but $(x, e) \neq (x, \gamma)$.
  Especially $\gamma.x_N \neq x_N$ for large N, 
  while $B_{N-l(\gamma)}(S,x_N) \rlisom B_{N-l(\gamma)}(S,\gamma.x_N)$ 
  and $d(x_N, \gamma.x_N)\leq  l(\gamma)$.
  By \cite[Proposition 2.3]{Elek_URS}, $Z$ is not generic.
\end{proof}

\begin{corollary}
	If $Z$ is generic, then $\Csr(Z)$ is simple.
\end{corollary}
This reproduces \cite[Theorem 7]{Elek_URS}.
\begin{proof}
	By Proposition \ref{prop:generic-principal}, 
    $\G$ is a minimal principal \etale{} groupoid,
    and every such groupoid has a simple reduced \Csalgebra{}
    by \cite[Proposition 4.3.7]{SimsNotes}.
\end{proof}

Regarding nuclearity, we are able to add
the converse direction to Elek's characterisation
\cite[Theorem 8]{Elek_URS}.
We first describe when our groupoids are amenable.

\begin{theorem}
  \label{thm:locPropANuclear}
	Let $Z$ be a uniformly recurrent subgroup 
    of the finitely generated discrete group $\Gamma$
    and let $\G$ be the groupoid associated with $Z$.
    The Schreier graph $\Schreier(H)$ of any group $H \in Z$
    has local property $A$,
	if and only if $\G$ is (topologically) amenable.
\end{theorem}
\begin{proof}
  We show that local property $A$ of the Schreier graph $\Schreier(H)$
  implies topological amenability of $\G$.
  Let $\G$ be constructed from $H\in Z$ 
  and let $\rho^n\colon \Gamma/H \rightarrow l^2(\Gamma/H)$ 
  implement local property $A$ of $S_\Gamma^Q(H)$.
  Let $S_n$ describe the locality of $\rho^n$ 
  as in Section \ref{sec:introduction}.
  Let $x_{S_n}$ denote any element of the equivalence class of $E_{S_n}$
  forming the $S_n$-component of $x$
  and define $f_n\in \C_c(\G)$ by $f_n(x, \gamma) = \rho^n_{x_{S_n}}(\gamma.x_{S_n})$.
  This is independent of the choice of $x_{S_n}$,
  as $\rho^n$ are locally defined and of width at most $S_n$, 
  continuous as it is constant on $2^{-S_n}$-balls and compactly supported,
  as it vanishes on the $2^{-S_n}$-ball of $\centralstar$.
  For any unit $u\in \G^0$ we have
  \begin{align*}
    \sum\limits_{(u, \gamma) \in \G^u} |f_n(u,\gamma)|^2
    &= \sum\limits_{y \in B_{S_n}(u_{S_n})} |\rho^n_{u_{S_n}}(y)|^2
    = \|\rho^n_{u_{S_n}}\|_2^2 = 1,
    \intertext{
      using that $\rho^n_{x}$ is supported in the $S_n$-ball centred at $u_{x}$.
      Furthermore, we calculate
    }
    f_n \ast f_n^\ast(x,\gamma)
    &= \sum\limits_{(\gamma.x, \gamma') \in \G^{\gamma.x}}
    f_n((x,\gamma)(\gamma.x, \gamma')) \overline{f_n}(\gamma.x, \gamma'),
    \intertext{ 
      where we may restrict to $l(\gamma')\leq S_n$, 
      as $f_n(\gamma.x, \gamma')$ vanishes otherwise,
    }
    &= \sum\limits_{(\gamma.x, \gamma') \in \G^{\gamma.x}, l(\gamma')\leq S_n}
    \rho^n_{x_{S_n + l(\gamma)}}(\gamma'\gamma x_{S_n + l(\gamma)}) 
    \overline{\rho^n_{\gamma x_{S_n + l(\gamma)}}}(\gamma'\gamma x_{S_n + l(\gamma)})\\
    &= \sum\limits_{z \in B_{S_n}(\gamma x_{S_n + l(\gamma)})}
    \rho^n_{x_{S_n + l(\gamma)}}(z) 
    \overline{\rho^n_{\gamma x_{S_n + l(\gamma)}}}(z)
    = \langle \rho^n_{x_{S_n + l(\gamma)}}, \rho^n_{\gamma x_{S_n + l(\gamma)}}\rangle,
  \end{align*}
  using again the assumptions on the support 
  of $\rho^n_{\gamma x_{S_n + l(\gamma)}}$.
  However, for fixed $u,v\in \Gamma/H$ we have
  \[
    \left| 1- \langle \rho^n_u, \rho^n_v\rangle\right| 
    = \left| \langle \rho^n_u - \rho^n_v, \rho^n_v\rangle \right|
    \leq \|\rho^n_u - \rho^n_v\|_2 \cdot \|\rho^n_v\|_2 
    \overset{(\dag)}\leq 1/n \cdot 1 
    \xrightarrow{n\to\infty} 0,
  \]
  where the estimate $(\dag)$ holds for large $n$ such that $d(u,v)\leq n$.
  But as $l(\gamma)$ is bounded on compact sets, 
  the same estimate may be used uniformly on any compact set for sufficiently large $n$,
  so that $f_n\ast f_n^\ast$ converges to one uniformly on compact subsets.
  Those functions witness the (topological) amenability of $\G$
  as in the original definition \cite[page 92]{Renault_GroupoidApproach}.

  Conversely, suppose that $\G$ is topologically amenable.
  By an equivalent characterisation of amenability \cite[Prop 2.2.13]{AmenableGrpd}, 
  we may assume that there is a sequence $f_n\in \C_c(\G)$, 
  such that
  \begin{align*}
    \sum\limits_{(u,\gamma)\in \G^u} |f_n(u,\gamma)|^2 &\xrightarrow{n\to\infty} 1 
    \text{ uniformly on compact subsets of $\G^0$ as a function of $u$}\\
    \text{and} 
    \sum\limits_{h\in \G^{r(g)}} |f_n(g^{-1}h) - f_n(h)|^2 &\xrightarrow{n\to\infty} 0 
    \text{ uniformly on compact subsets of $\G$ as a function of $g$.}
  \end{align*}
  To obtain local maps $\rho^n_x$ 
  as in the definition of local property A,
  we first, for any $\eps>0$,
  approximate $f_n$ by (uniformly) locally constant functions $f_{n,\eps}\in \C_c(\G)$
  such that $\|f_n - f_{n,\eps}\|_\infty < \eps$, 
  choosing $f_{n,\eps}$ to be zero on the largest possible ball 
  centred at $\centralstar$.

  We next construct a sequence $\eps_n \searrow 0$, 
  such that $f_{n,\eps_n}$ satisfies the convergence properties above:
  As for fixed $n$ there is $T_n\in N$ 
  such that both $f_n$ and $f_{n,\eps}$
  for every $\eps>0$ 
  vanish on the $2^{-T_n}$-ball of $\centralstar$ 
  we may restrict summation to arrows
  described by group elements $\gamma$
  with length $l(\gamma)$ at most $T_n$,
  such that the respective sums become finite
  with less than $(|Q| + 1)^{T_n}$ terms.
  Therefore,
  \begin{align*}
    \left|\sum\limits_{(u,\gamma)\in \G^u} 
      |f_n(u,\gamma)|^2 - \sum\limits_{(u,\gamma)\in \G^u} 
      |f_{n,\eps}(u,\gamma)|^2\right|
      &\leq \sum\limits_{\substack{(u,\gamma)\in \G^u,\\l(\gamma)\leq T_n}}
      (|f_n(u,\gamma)| + \eps)^2 - |f_{n}(u,\gamma)|^2\\
    &\leq (|Q|+1)^{T_n}(2\|f_n\|_\infty + \eps)\eps 
    \xrightarrow{\eps\to 0} 0,
  \intertext{
  with the convergence uniformly in $u \in \G^0$.
  Likewise, we calculate
  }
    \sum\limits_{h\in \G^{r(g)}} |f_{n,\eps}(g^{-1}h) - f_{n,\eps}(h)|^2
    &\leq \sum\limits_{h\in \G^{r(g)}, l(\gamma)\leq T_n} (|f_{n}(g^{-1}h) - f_{n}(h)|+2\eps)^2\\
    &\leq (4\eps^2 + 8\eps \|f_n\|_\infty)(|Q|+1)^{T_n} 
    + \sum\limits_{h\in \G^{r(g)}} |f_{n}(g^{-1}h) - f_{n}(h)|^2 
    \\
    &\xrightarrow{\eps \to 0}
    \sum\limits_{h\in \G^{r(g)}} |f_{n}(g^{-1}h) - f_{n}(h)|^2,
  \intertext{
  with the convergence uniformly in $g\in \G$.
  Picking $\eps_n\to 0$ such that 
  }
  1/n &\geq
  (|Q|+1)^{T_n}(2\|f_n\|_\infty + \eps_n)\eps_n 
  \\
  \text{and}\qquad
  1/n &\geq
  (4\eps_n^2 + 8\eps_n \|f_n\|_\infty)(|Q|+1)^{T_n} 
  \end{align*}
  does the trick.
  To such $f_{n,\eps_n}$ 
  we may assign $\rho^{n} \colon \Gamma/H \rightarrow l^2(\Gamma/H)$ 
  by $\rho^{n}_x(\gamma x) = f_{n,\eps_n}([[x]], \gamma)$.

  Picking $S_{n}$ such that $f_{n,\eps_n}$ is  constant on any $2^{-S_{n}}$-ball,
  we see that $\rho^{n}_x$ is supported on $B_{S_{n}}(x)$, 
  since $f_{n,\eps_n}$ vanishes on the $2^{-S_{n}}$-neighbourhood 
  of $\centralstar$.
  Similarly,
  $\rho^{n}_x(\gamma x) = \rho^{n}_y(\gamma y)$ 
  for $l(\gamma)\leq S_{n}$ if the $S_{n}$-balls around $x$ and $y$ are isomorphic, 
  as $([[x]],\gamma)$ and $([[y]], \gamma)$ are $2^{-S_{n}}$-close in that case
  and we conclude that $\rho_n$ is local
  of width $S_n$.

  Furthermore, we compute that
  \begin{align*}
    \|\rho^{n}_{\gamma H}\|_2^2 
    &= \sum\limits_{\gamma'H\in \G/H} |\rho^{n}_{\gamma H}(\gamma'H)|^2
    = \sum\limits_{([[\gamma H]],\gamma')\in \G^u} |f_{n,\eps_n}([[\gamma H]],\gamma')|^2 
    \xrightarrow{n\to\infty} 1
    \intertext{converges uniformly on $\G^0$. Likewise,}
    \|\rho^{n}_{\gamma H} - \rho^{n}_{\gamma'H}\|_2^2
    &= \sum \limits_{\eta H \in \Gamma/H} |\rho^{n}_{\gamma H}(\eta H) - \rho^{n}_{\gamma'H}(\eta H)|^2\\
    &= \sum\limits_{\eta H \in \Gamma/H} | f_{n,\eps_n}(\gamma[[H]], \eta\gamma^{-1}) - f_{n,\eps_n}(\gamma'[[H]], \eta(\gamma')^{-1})|^2\\
    &= \sum\limits_{\eta H \in \Gamma/H} | f_{n,\eps_n}\left((\gamma[[H]], \gamma'\gamma^{-1})(\gamma'[[H]], \eta(\gamma')^{-1})\right) - f_{n,\eps_n}(\gamma'[[H]], \eta(\gamma')^{-1})|^2\\
    &= \sum\limits_{\eta H \in \Gamma/H} | f_{n,\eps_n}\left((\gamma'[[H]], \gamma(\gamma')^{-1})^{-1}(\gamma'[[H]], \eta(\gamma')^{-1})\right) - f_{n,\eps_n}(\gamma'[[H]], \eta(\gamma')^{-1})|^2\\
    &= \sum\limits_{h \in \G^{\gamma'[[H]]}} | f_{n,\eps_n}\left((\gamma[[H]], \gamma'\gamma^{-1})h\right) - f_{n,\eps_n}(h)|^2
    \xrightarrow {n\to\infty} 0
  \end{align*}
  converges uniformly for $(\gamma'[[H]], \gamma(\gamma')^{-1})$ in compact subsets of $\G$. 
  In particular, the convergence is uniform, 
  when the pair $(\gamma, \gamma')$ is taken from a subset with bounded difference, 
  that is, if there is $N>0$ such that
  $l(\gamma(\gamma')^{-1}) < N$.
  More importantly, 
  this means that
  for any $N\in \N$ 
  we may choose $n_0\in \N$, 
  such that $\|\rho^n_{\gamma H} - \rho^n_{\gamma'H}\|_2^2\leq 1/N$, 
  whenever $d_{\Schreier(H)}(\gamma H, \gamma'H) \leq N$ 
  and $n\geq n_0$.
  The analogous statement holds after replacing $\rho^n_x$
  with the normed function $\hat\rho^n_x = \rho^n_x/\|\rho^n_x\|_2$, since
  \begin{align*}
    \|\hat\rho^n_{\gamma H} - \hat \rho^n_{\gamma' H}\|_2
    &\leq \frac{1}{\|\rho^n_{\gamma H}\|_2}
    \left( 
      \|\rho^n_{\gamma H} - \rho^n_{\gamma'H}\|_2 
      + \left| \|\rho^n_{\gamma H}\|_2 - \|\rho^n_{\gamma' H}\|_2\right|
    \right),
  \end{align*}
  while $\|\rho^n_{\gamma H}\|_2$ 
  and $ \left| \|\rho^n_{\gamma H}\|_2 - \|\rho^n_{\gamma' H}\|_2\right|$
  converge uniformly to $1$ and $0$, respectively.
  Then, after relabelling, 
  $\hat\rho^n$ witnesses local property A of $\Schreier(H)$.
\end{proof}

For \etale{} groupoids there is a clear relation
between amenability and nuclearity of their reduced \Csalgebra{}s,
which directly translates to our case.
Compare for example \cite[Section 2]{Sims_AmenableFell} 
for a brief overview of the different notions of groupoid amenability.

\begin{corollary}
  \label{cor:locPropANuclear}
  Let $Z$ be a uniformly recurrent subgroup and $H\in Z$.
  The graph $\Schreier(H)$ has local property A
  if and only if
  the \Csalgebra{} $\Csr(Z)$ is nuclear.
\end{corollary}
This reproduces  and extends \cite[Theorem 8]{Elek_URS}.
\begin{proof}
  As the groupoid $\G$ associated with $Z$ is \etale, 
  $\Csr(\G)$ is nuclear if and only if 
  $\G$ is (topologically) amenable by \cite[Corollary 6.2.14]{AmenableGrpd}.
  By Theorem \ref{thm:locPropANuclear}
  this is exactly the case if the Schreier graph $\Schreier(H)$
  has local property $A$.
\end{proof}

\begin{remark}
  A minimal \emph{amenable} second countable \etale{} groupoid 
  has simple reduced \Csalgebra, 
  if and only if it is topologically principal.
  Hence, if the Schreier graphs of $Z$ have local property $A$, 
  then $\Csr(Z)$ is simple if and only if $Z$ is generic.

  More generally, 
  $\Csr(Z)$ is simple 
  for arbitrary URS $Z$,
  if and only if $\Csr(\G)$ has the \emph{intersection property} 
  relative to $\C_0(\G^0)$,
  meaning that every (nontrivial) ideal of $\Csr(\G)$
  intersects $\C_0(\G^0)$ nontrivially.
\end{remark}

\bibliographystyle{amsplain}
\bibliography{meta/nonmathsci.bib,meta/mathsci.bib}
\vspace{1cm}
\begin{minipage}[l]{\textwidth}
\noindent Clemens Borys\\
Department of Mathematical Sciences\\
University of Copenhagen\\ 
Universitetsparken 5, DK-2100, Copenhagen\\
Denmark \\
borys@math.ku.dk
\end{minipage}

\end{document}